\documentclass[12pt, a4paper]{article}
\usepackage{amsfonts,amssymb,amsmath,amscd,latexsym,makeidx,theorem}
\usepackage{hyperref}
\usepackage{color}

\newtheorem{thm}{Theorem}[section]
\newtheorem{lem}[thm]{Lemma}
\newtheorem{prop}[thm]{Proposition}

\newtheorem{defn}{Definition}[section]

\newenvironment{proof}{\noindent\emph{Proof.}}{\hfill$\square$\medskip}

\newcommand{\s}{\mathcal{S}}
\newcommand{\R}{\mathbb{R}}
\newcommand{\e}{\varepsilon}
\newcommand{\D}{(-\Delta)}
\newcommand{\F}{\mathcal{F}}

\title{Moser functions and fractional Moser-Trudinger type inequalities}

\author{Ali Hyder \thanks{The author is supported by the Swiss National Science Foundation, project nr. PP00P2-144669.} \\ {\small Universit\"at Basel}\\
{\small \texttt{ali.hyder@unibas.ch} }}

\begin{document}

\maketitle
\begin{abstract}
We improve the sharpness of some fractional Moser-Trudinger type inequalities, particularly those
studied by Lam-Lu and  Martinazzi. As an application, improving upon works of Adimurthi and Lakkis, we prove the existence of weak 
solutions to the problem
$$\D^\frac{n}{2}u=\lambda ue^{bu^2} \,\text{ in }\Omega,\, 0<\lambda<\lambda_1,\,b>0,$$
with Dirichlet boundary condition, for any domain $\Omega$ in $\R^n$ with finite measure. Here $\lambda_1$ is the first eigenvalue of $(-\Delta)^\frac n2$ on $\Omega$.
\end{abstract}
\section{Introduction to the problem}
Let $n\geq 2$ and let $\Omega$ be a bounded domain in $\R^n$. The Sobolev embedding theorem states that $W^{k,p}_0(\Omega)\subset L^q(\Omega)$ for 
$1\leq q\leq \frac{np}{n-kp}$ and $kp<n$. However, it is not true that $W^{k,p}_0(\Omega)\subset L^\infty(\Omega)$ for $kp=n$.
In the borderline case, as shown by Yudovich \cite{VIY}, Pohozaev \cite{Poho} and Trudinger \cite{Trudinger}, $W^{1,n}_0(\Omega)$ embeds into 
an Orlich space and in fact 
\begin{align} \label{1}
 \sup_{u\in W^{1,n}_0(\Omega),\,\|\nabla u\|_{L^n(\Omega)}\leq 1}\int_{\Omega}e^{\alpha|u|^\frac{n}{n-1}}dx<\infty,
\end{align}
for some $\alpha>0$. Moser \cite{Moser} found the best constant $\alpha$ in the inequality \eqref{1}, obtaining the so called Moser-Trudinger inequality: 
\begin{align} \label{2}
 \sup_{u\in W^{1,n}_0(\Omega),\,\|\nabla u\|_{L^n(\Omega)}\leq 1}\int_{\Omega}e^{\alpha_n|u|^\frac{n}{n-1}}dx<\infty, \quad \alpha_n=n|S^{n-1}|^\frac{1}{n-1}.
 \end{align}
 The constant $\alpha_n$ in \eqref{2} is the best constant in the sense that for any $\alpha>\alpha_n$, the supremum in \eqref{1} is infinite. 
A generalized version of Moser-Trudinger inequality is the following theorem of  Adams \cite{Adams}:
 
 \medskip
\noindent \textbf{Theorem A (\cite{Adams})}
\emph{
If $k$ is a positive integer less than $n$, then there is a constant $C=C(k,n)$ such that
$$\sup_{u\in C^k_c(\Omega),\,\|\nabla^ku\|_{L^\frac nk(\Omega)}\leq 1}\int_{\Omega}e^{\alpha|u|^\frac{n}{n-k}}dx\leq C|\Omega|,$$
where
$$\alpha=\alpha(k,n)=\frac{n}{|S^{n-1}|}\left\{\begin{array}{ll}
               \bigg[\frac{\pi^\frac n2 2^k\Gamma\left(\frac{k+1}{2}\right)}{\Gamma\left(\frac{n-k+1}{2}\right)}\bigg]^\frac{n}{n-k}, &m= odd,\\
                   \rule{.0cm}{.8cm}
                   \bigg[\frac{\pi^\frac n2 2^k\Gamma\left(\frac{k}{2}\right)}{\Gamma\left(\frac{n-k}{2}\right)}\bigg]^\frac{n}{n-k}, &m= even,
                                         \end{array}
 \right.$$
and $\nabla^k:=\nabla\Delta^\frac{k-1}{2}$ for $k$ odd and $\nabla^k =\Delta^\frac{k}{2}$ for $k$ even. Moreover the constant $\alpha$ is sharp in the sense that
\begin{align}\label{infty}
 \sup_{u\in C^k_c(\Omega),\,\|\nabla^ku\|_{L^\frac nk(\Omega)}\leq 1}\int_{\Omega}f(|u|)e^{\alpha|u|^\frac{n}{n-k}}dx=\infty,
\end{align}
for any $f:[0,\infty)\to[0,\infty)$ with $\lim_{t\to\infty}f(t)=\infty$.\footnote{Identity \eqref{infty} is proven in \cite{Adams}, although not explicitly stated.}}

\medskip
In a recent work Martinazzi \cite{LM} has studied the Adams inequality in a fractional setting. In order to state its result first we recall the space 
$$L_s(\R^n):=\left\{u\in L^1_{loc}(\R^n): \int_{\R^n}\frac{|u(x)|}{1+|x|^{n+2s}}dx<\infty\right\}.$$ The operator $\D^s$ can be defined on the space 
$L_s(\R^n)$ via the duality 
\begin{align}\label{welldefined}
 \langle\D^su,\,\varphi\rangle:=\int_{\R^n}u\D^s\varphi dx,\quad \varphi\in\s(\R^n),
\end{align}
where $$\D^s\varphi=\F^{-1}\left(|\xi|^{2s}\hat{\varphi}\right), \quad \varphi\in\s(\R^n),$$
$\F$ is the normalized Fourier transform and $\s(\R^n)$ is the Schwartz space.
Notice that the integral in \eqref{welldefined} is well-defined thanks to \cite[Proposition 2.1]{H-Structure}.

Now for an open set 
$\Omega\subseteq\R^n$ (possibly $\Omega=\R^n$), $s>0$ and $1\leq p\leq \infty$ we define the fractional Sobolev space  $\tilde{H}^{s,p}(\Omega)$ by
$$\tilde{H}^{s,p}(\Omega):=\left\{u\in L^p(\Omega): u=0\text{ on }\R^n\setminus\Omega,\,(-\Delta)^\frac{s}{2}u\in {L^p(\R^n)}\right\}.$$

\medskip
\noindent \textbf{Theorem B (\cite{LM})}
\emph{For any open set $\Omega\subset\R^n$ with finite measure and for any $p\in (1,\,\infty)$ we have 
$$\sup_{u\in \tilde{H}^{\frac np,p}(\Omega),\,\|(-\Delta)^\frac{n}{2p}u\|_{L^p(\Omega)}\leq 1}
\int_{\Omega}e^{\alpha_{n,p}|u|^{p'}}dx\leq C_{n,p}|\Omega|,$$
where the constant $\alpha_{n,p}$ is given by 
\begin{align}\label{alpha-np}
 \alpha_{n,p}=\frac{n}{|S^{n-1}|}\left(\frac{\Gamma(\frac {n}{2p})2^\frac np\pi^\frac n2}{\Gamma(\frac{np-n}{2p})}\right)^{p'}.
\end{align}
Moreover, the constant $\alpha_{n,p}$ is sharp in the sense that we cannot replace it with any larger one without making the above supremum infinite.}

\medskip

Notice that condition \eqref{infty} in Theorem A is sharper than only requiring that the constant $\alpha$ in the exponential is sharp, as  done in Theorem B.
In fact Martinazzi asked whether it is true that 
  \begin{align}\label{infty1}
  \sup_{u\in \tilde{H}^{\frac np,p}(\Omega),\,\|(-\Delta)^\frac{n}{2p}u\|_{L^p(\Omega)}\leq 1}
  \int_{\Omega}f(|u|)e^{\alpha_{n,p}|u|^{p'}}dx=\infty,
 \end{align}
for any $f:[0,\infty)\to[0,\infty)$ with
\begin{equation}\label{condf}
\lim_{t\to\infty}f(t)=\infty,\quad\text{$f$ is Borel measurable},
\end{equation}
and  $\alpha_{n,p}$ is given by \eqref{alpha-np}.

The point here is that Adams constructs smooth and compactly supported test functions similar to the standard Moser functions (constant in a small ball, 
and decaying logarithmically on an annulus), and then he estimates their $H^{k,\frac{n}{k}}_0$-norms in a very precise way. 
This becomes much more delicate when $k$ is not integer because instead of computing partial derivatives, one has to estimate the norms of fractional
Laplacians (the term $\|(-\Delta)^\frac{n}{2p}u\|_{L^p(\Omega)}$ in \eqref{infty1}). This is indeed done in \cite{LM}, but the test functions introduced
by Martinazzi are not efficient enough to prove \eqref{infty1}. As we shall see this has consequences for applications to PDEs.

 We shall prove that the answer to Martinazzi's question is positive, indeed in a slightly stronger form, namely the supremum in \eqref{infty1} is 
infinite even if we consider the full  $H^{\frac{n}{p},p}$-norm on the whole space. More precisely we have:

\begin{thm}\label{main-thm}
Let $\Omega$ be an open set in $\R^n$ with finite measure and let $f:[0,\infty)\to[0,\infty)$ satisfy \eqref{condf}. Then 
$$\sup_{u\in \tilde{H}^{\frac np,p}(\Omega),\,\|u\|^p_{L^p(\Omega)}+\|(-\Delta)^\frac{n}{2p}u\|_{L^p(\R^n)}^p\leq 1}
\int_{\Omega}f(|u|)e^{\alpha_{n,p}|u|^{p'}}dx=\infty,\quad 1<p<\infty,$$ where the constant $\alpha_{n,p}$ is given by \eqref{alpha-np}.
\end{thm}

The main difficulty in the proof of Theorem \ref{main-thm} is to construct test and cut-off functions in a way that their fractional Laplacians of 
suitable orders can be estimated precisely. This will be done in section \ref{section2}.

Here we mention that using  a Green's representation formula, Iula-Maalaoui-Martinazzi \cite{IIM} proved a particular case of 
 Theorem \ref{main-thm} in one dimension. Their proof, though, does not extend to spaces $\tilde H^{\frac{n}{p},p}(\Omega)$ when $\frac{n}{p}>1$ because the function constructed using the Green representation formula do not enjoy enough smoothness at the boundary. Trying to solve this with a smooth cut-off function at the boundary allows to prove \eqref{infty1} only when $f$ grows fast enough at infinity (for instance $f(t)\ge t^{a}$ for some $a>p'$).
 
 \medskip
 
 Now we move to Moser-Trudinger type inequalities on domains with infinite measure.  In this direction we refer to \cite{Ruf, Lam-Lu, Mas} and the 
 references there in.  For our purpose, here we only state the work of Lam-Lu \cite{Lam-Lu}.
  \medskip
 
 \noindent\textbf{Theorem C} (\cite{Lam-Lu})
 \emph{Let $p\in (1,\infty)$ and $\tau>0$. Then for every domain $\Omega\subset\R^n$ with finite measure, there exists $C=C(n,p,\tau)>0$ such that
 $$\sup_{u\in \tilde{H}^{\frac np,p}(\R^n),\,\|(\tau I-\Delta)^\frac{n}{2p}u\|_{L^p(\R^n)}\leq 1}
\int_{\Omega}e^{\alpha_{n,p}|u|^{p'}}dx\leq C|\Omega|,$$
and 
 $$\sup_{u\in \tilde{H}^{\frac np,p}(\R^n),\,\|(\tau I-\Delta)^\frac{n}{2p}u\|_{L^p(\R^n)}\leq 1}
\int_{\R^n}\Phi(\alpha_{n,p}|u|^{p'})dx<\infty,$$
where $\alpha_{n,p}$ is given by \eqref{alpha-np} and
$$\Phi(t):=e^t-\sum_{j=0}^{j_p-2}\frac{t^j}{j!},\quad j_p:=\min\{j\in\mathbb{N}:j\geq p\}.$$ Furthermore, the constant $\alpha_{n,p}$ is sharp in the 
above inequalities, i.e., if $\alpha_{n,p}$ is replaced by any $\alpha>\alpha_{n,p}$, then the supremums are infinite.
 }

\medskip

In the spirit of Theorem \ref{main-thm} we prove a stronger version of the sharpness of the constant in Theorem C.

 \begin{thm}\label{thm-2}
 Let $\Omega\subset\R^n$ be a domain with finite measure and let $f:[0,\infty)\to[0,\infty)$ satisfy \eqref{condf}. Then
 for any  $\tau>0$ and for any $p\in(1,\infty)$ we have (with the notations as in Theorem C)
 $$\sup_{u\in \tilde{H}^{\frac np,p}(\R^n),\,\|(\tau I-\Delta)^\frac{n}{2p}u\|_{L^p(\R^n)}\leq 1}
\int_{\Omega}f(|u|)e^{\alpha_{n,p}|u|^{p'}}dx=\infty,$$ 
and 
$$\sup_{u\in \tilde{H}^{\frac np,p}(\R^n),\,\|(\tau I-\Delta)^\frac{n}{2p}u\|_{L^p(\R^n)}\leq 1}
\int_{\R^n}f(|u|)\Phi(\alpha_{n,p}|u|^{p'})dx=\infty.$$
 \end{thm}

As an application of Theorem \ref{main-thm} (in the case $p=2$ and $f(t)=t^2$, compare to \eqref{estimates} below) we prove the existence of (weak) solution to a semilinear elliptic equation with exponential nonlinearity.
In order to state the theorem first we need the following definition. 

 \begin{defn}
 Let $\Omega$ be an open set  in $\R^n$ with finite measure.  Let $f\in L^{p}(\Omega)$ for some $p\in(1,\infty)$. We say that $u$ is a weak solution of 
 \begin{align}\notag
   \D^\frac{n}{2}u=f \,\text{ in }\Omega, 
  \end{align}
  if $u\in \tilde{H}^{\frac n2,2}(\Omega)$ satisfies 
  $$\int_{\R^n}\D^\frac n4u\D^\frac n4v dx=\int_{\Omega}fvdx\quad\text{for every }v\in\tilde{H}^{\frac n2,2}(\Omega).$$
 \end{defn}

 \begin{thm}\label{thm-3} 
 Let $\Omega$ be an open set in $\R^n$ with finite measure.
  Let $0<\lambda<\lambda_1$ and $b>0$. Then there exists a nontrivial weak solution to the problem 
  \begin{align}\label{eqn}
   \D^\frac{n}{2}u=\lambda ue^{bu^2} \,\text{ in }\Omega.
  \end{align}
 \end{thm}

 Due to the fact that the embedding $\tilde{H}^{\frac n2,2}(\Omega)\hookrightarrow L^2(\Omega)$ is compact for any open set $\Omega$ with finite
 measure (see Lemma \ref{compact} in Appendix), we do not need any regularity assumption or boundedness assumption on the domain $\Omega$.

The equation \eqref{eqn} has been well studied by several authors in even and odd dimensions, with emphasis both on existence and
compactness properties see e.g. \cite{Adi-Str, dru, IS, Lakkis, Luca-Armin, MM, mar4, MS2, RS, str2}. For instance, Lakkis \cite{Lakkis}, 
extending a work of Adimurthi \cite{Adi}, proved the existence 
 of solution to \eqref{eqn} in any even dimension. In a recent work Iannizzotto-Squassina \cite{IS} have proven existence of
 nontrivial weak solution of \eqref{eqn} with $\Omega=(0,1)$ under an assumption, which turns out to be satisfied thanks
 to our Theorem \ref{main-thm}, applied with $p=2$ (see Lemma \ref{Ian-Squa}).
 
 \section{Moser type functions and proof of Theorems \ref{main-thm}, \ref{thm-2}}\label{section2}
 We construct Moser type functions as follows:
 
 First we fix two smooth functions $\eta$ and $\varphi$ such that $0\leq \eta,\,\varphi\leq 1$,
 $$\eta\in C_c^\infty(-1,1),\quad\eta=1 \text{ on }(-\frac34,\frac34),$$ and
 $$\varphi\in C_c^\infty ((-2,2)),\quad \varphi=1,\text{ on }(-1,1).$$
For $\varepsilon>0$, we set         
$$\psi_\varepsilon(t)=\left\{\begin{array}{ll}
                                                      1-\varphi_\varepsilon (t) &\text{ if }0\leq t\leq\frac12\\
                                                   \eta(t) &\text{ if }t\geq \frac12,
                                                      \end{array}\right.$$       
    and 
     \begin{align}
 v_\varepsilon(x)=\left(\log\frac 1\varepsilon\right)^{-\frac{1}{p}}\left(\log \left(\frac1\varepsilon\right)\varphi_\varepsilon(|x|)+
  \log\left(\frac{1}{|x|}\right)\psi_\varepsilon(|x|)\right)\quad x\in \R^n,\notag
\end{align}
     where $$\varphi_\varepsilon(t)=\varphi(\frac t\varepsilon).$$
     
     Our aim is to show that  the supremums (in Theorems \ref{main-thm} and \ref{thm-2}) taken over the functions
    $\{v_\e\}_{\e>0}$ (up to a proper normalization) are infinite.
    
  The following proposition is crucial in the proof of Theorem \ref{main-thm}.     
 \begin{prop} \label{prop1}
Let $$u_{\e}(x):=|S^{{n-1}}|^{-\frac1p}2^{\frac{n}{p'}}\pi^{\frac n2}\Gamma(\frac{n}{2p'})\frac{1}{\Gamma(\frac{n}{2p})\gamma_{n}}v_{\e}(x).$$
 Then for $1<p<\infty$ there exists a constant $C>0$ such that
  $$\|\D^\frac{n}{2p}u_\e\|_{L^p(\R^n)}\leq \left(1+C\left(\log\frac{1}{\e}\right)^{-1}\right)^\frac 1p.$$
 \end{prop}
\begin{proof}
Since the proof of above proposition is quite trivial if $\frac{n}{2p}$ is an integer, from now on we only consider the case when $\frac{n}{2p}$ is not an integer.

 From Lemmas \ref{V-near-zero} and \ref{V-epsilon} (below) we have 
 $$\|\D^\frac{n}{2p}u_\e\|_{L^p(B_{3\e}\cup B_2^c )}^p\leq C\left(\log\frac{1}{\e}\right)^{-1}.$$
In order to estimate $\D^\sigma v_\e$ on the domain $\left\{x:3\e<|x|<2\right\}$ we consider the function 
\begin{align}
R_\e(x)=v_\varepsilon(x)- \left(\log\frac 1\varepsilon\right)^{-\frac{1}{p}}\log\frac{1}{|x|}=:f_\e(x)+g_\e(x)\quad x\in\R^n, \notag
\end{align}
 where
\begin{align}
f_\varepsilon(x):&=\left\{\begin{array}{ll}\notag
                            v_\varepsilon(x)- \left(\log\frac 1\varepsilon\right)^{-\frac{1}{p}}\log\frac{1}{|x|} &\text{ if }|x|<2\varepsilon\\
                            0                  &\text{ if }|x|\geq 2\varepsilon,
                           \end{array}\right.\\
                &= \left(\log\frac 1\e\right)^{-\frac{1}{p}}\left(\log\frac 1\e-\log\frac{1}{|x|}\right)\varphi_\e(|x|)       \notag   
\end{align}
and 
\begin{align}
g_\varepsilon(x):&=\left\{\begin{array}{ll}\notag
                            v_\varepsilon(x)- \left(\log\frac 1\varepsilon\right)^{-\frac{1}{p}}\log\frac{1}{|x|} &\text{ if }|x|>\frac12\\
                            0                  &\text{ if }|x|\leq \frac12
                           \end{array}\right.\\
                 &=\left(\log\frac 1\e\right)^{-\frac{1}{p}}\left(\eta(|x|)-1\right)\log\frac{1}{|x|}.\notag
\end{align}
It is easy to see that for any $\sigma>0$ 
 \begin{align}\label{est-g}
  \sup_{x\in\R^n}|(-\Delta)^\sigma g_\varepsilon(x)|\leq C\left(\log\frac 1\varepsilon\right)^{-\frac{1}{p}}.
 \end{align}
With the help of Lemma \ref{log} and the triangle inequality we bound
\begin{align}
 |\D^\frac{n}{2p}u_\e(x)|&=\frac{1}{|S^{n-1}|^\frac1p\beta_{n,\frac{n}{2p}}}\left|\D^\frac{n}{2p}R_\e(x)+\left(\log\frac 1\e\right)^{-\frac 1p}
\D^\frac{n}{2p}\log\frac{1}{|x|}\right|\notag\\
&\leq C|\D^\frac{n}{2p}R_\e(x)|+\left(\log\frac 1\e\right)^{-\frac 1p}\frac{1}{|S^{n-1}|^\frac1p}\frac{1}{|x|^\frac np}.\notag
\end{align}
Using the elementary inequality 
$$(a+b)^q\leq a^q+C_q(b^q+a^{q-1}b),\quad 1\leq q<\infty,\,a\geq 0,\,b\geq 0,$$
we get
\begin{align}
 &\int_{3\e<|x|<2}|\D^\frac{n}{2p}u_\e(x)|^pdx\notag\\
 &\leq \int_{3\e<|x|<2}\left(\log\frac 1\e\right)^{-1}\frac{1}{|S^{n-1}|}\frac{1}{|x|^n} dx+
             C\int_{3\e<|x|<2}|\D^\frac{n}{2p}R_\e(x)|^pdx\notag\\
 &\quad+C\left(\log\frac 1\e\right)^{-\frac{1}{p'}}\int_{3\e<|x|<2}\frac{1}{|x|^\frac{n}{p'}}|\D^\frac{n}{2p}R_\e(x)|dx\notag\\
  &\leq 1+C\left(\log\frac 1\e\right)^{-1}+C\left(\log\frac 1\e\right)^{-\frac{1}{p'}}\int_{3\e<|x|<2}\frac{1}{|x|^\frac{n}{p'}}|\D^\frac{n}{2p}R_\e(x)|dx, \notag
\end{align} 
where the last inequality follows from Lemma \ref{f-epsilon} (below). 
Using the pointwise estimate  in Lemma \ref{f-epsilon} and \eqref{est-g} one can show that 
$$\int_{3\e<|x|<2}\frac{1}{|x|^\frac{n}{p'}}|\D^\frac{n}{2p}R_\e(x)|dx\leq C\left(\log\frac 1\e\right)^{-\frac 1p},$$
which completes the proof.
\end{proof}

\begin{lem}\label{V-near-zero}
 Let $p\in(1,\infty)$. Then there exists a constant $C=C(n,p,\sigma)>0$ such that 
 $$|(-\Delta)^{\sigma}v_\e(x)|\leq C\left(\log\frac 1\varepsilon\right)^{-\frac{1}{p}}\e^{-2\sigma}\quad \text{for }|x|\leq3\e,\, 0<\sigma<\frac n2.$$
 Moreover,
 $$\|\D^\frac{n}{2p} v_\e\|^p_{L^p(B_{3\e})}\leq C\left(\log\frac 1\e\right)^{-1}.$$
\end{lem}
\begin{proof}
 We claim that for every nonzero multiindex $\alpha\in\mathbb{N}^n$ there exists $C=C(n,\alpha)>0$ such that 
 \begin{align}
  |D^\alpha v_\e(x)|\leq C\left(\log\frac 1\varepsilon\right)^{-\frac{1}{p}}\e^{-|\alpha|},\quad x\in \R^n.\label{ve}
 \end{align}

 The claim follows from the fact that $D^\alpha (\varphi_\e+\psi_\e)=0$ on $B_\frac12$ and hence we have the lemma if $\sigma$ is an integer.
 In the case when $\sigma$ is not a integer then we write $\sigma=m+s$ where $0<s<1$ and $m$ is an nonnegative integer. Then for $|x|\leq 3\e$ we have
  (the following equivalent definition of fractional Laplacian can be found in \cite{Sil, Valdinoci})
 \begin{align}
(-\Delta)^{\sigma}v_\e(x)=C_{n,s}\int_{\R^n}\frac{\D^mv_\e(x+y)+\D^mv_\e(x-y)-2\D^mv_\e(x)}{|y|^{n+2s}}dy.\notag
 \end{align}
$$A_1=\left\{x:|x|\leq 2\e\right\},\quad A_2=\left\{x:2\e<|x|\leq \frac14\right\}\text{ and }A_3=\left\{x:|x|>\frac14\right\},$$ we have
$$(-\Delta)^{\sigma}v_\e(x)=C_{n,s}\sum_{i=1}^3I_i,$$ where $$I_i:=\int_{A_i}\frac{\D^mv_\e(x+y)+\D^mv_\e(x-y)-2\D^mv_\e(x)}{|y|^{n+2s}}dy.$$
For $y\in A_1$, using \eqref{ve} we have
\begin{align}
|\D^mv_\e(x+y)+\D^mv_\e(x-y)-2\D^mv_\e(x)|&\leq|y|^2\|D^2\D^m v_\e\|_{L^\infty}\notag\\
&\leq C|y|^2\e^{-2m-2}\left(\log\frac 1\varepsilon\right)^{-\frac{1}{p}},\notag
\end{align}
and hence
\begin{align}
 |I_1|\leq C\e^{-2m-2}\left(\log\frac 1\varepsilon\right)^{-\frac{1}{p}}\int_{A_1}\frac{dy}{|y|^{n+2s-2}}\leq C\left(\log\frac 1\e\right)^{-\frac{1}{p}}\e^{-2\sigma}.\notag
\end{align}
For $m\geq 1$, again by \eqref{ve}
\begin{align}
& |\D^mv_\e(x+y)-\D^mv_\e(x)|
  \leq C\left(\log\frac{1}{\e}\right)^{-\frac1p}\e^{-2m}.\notag
\end{align}
Therefore,
\begin{align}
 |I_2+I_3|\leq C\left(\log\frac{1}{\e}\right)^{-\frac1p}\e^{-2m}\int_{|y|>\e}\frac{dy}{|y|^{n+2s}}\leq C\left(\log\frac 1\e\right)^{-\frac{1}{p}}\e^{-2\sigma}.\notag
\end{align}
  Since on $A_2$ $|x+y|\leq 3\e+\frac 14<\frac12$, one has
\begin{align}
 &\left(\log\frac{1}{\e}\right)^{\frac1p}|v_\e(x+y)-v_\e(x)|\notag\\
 &= \left|\log\left(\frac1\e\right)\left(\varphi_\e(|x+y|)+\psi_\e(|x+y|)-\varphi_\e(|x|)- \psi_\e(|x|)\right)\right.\notag\\
& \quad \left.+\log\left(\frac{\e}{|x+y|}\right)\psi_\e(|x+y|)-\log\left(\frac{\e}{|x|}\right)\psi_\e(|x|)\right|\notag\\
&=\left|\log\left(\frac{\e}{|x+y|}\right)\psi_\e(|x+y|)-\log\left(\frac{\e}{|x|}\right)\psi_\e(|x|)\right|\notag\\
&\leq C+\left|\log\left(\frac{\e}{|x+y|}\right)\psi_\e(|x+y|)\right|.\notag
\end{align}
Hence, for $m=0$, changing the variable $y\mapsto \e z$
\begin{align}
 |I_2|&\leq C\left(\log\frac{1}{\e}\right)^{-\frac1p}\e^{-2s}+
  C\left(\log\frac{1}{\e}\right)^{-\frac1p}\int_{\e<|y|<\frac14}\frac{\left|\log\left(\frac{\e}{|x+y|}\right)\psi_\e(|x+y|)\right|}{|y|^{n+2s}}dy\notag\\
  &\leq C\left(\log\frac{1}{\e}\right)^{-\frac1p}\e^{-2s}+ C\left(\log\frac{1}{\e}\right)^{-\frac1p}\e^{-2s}
    \int_{|z|>1}\frac{\left|\log|\frac x\e+z|\right|\psi_\e(\e|\frac x\e+z|)}{|z|^{n+2s}}dz\notag\\
  &\leq C\left(\log\frac{1}{\e}\right)^{-\frac1p}\e^{-2s}+ C\left(\log\frac{1}{\e}\right)^{-\frac1p}\e^{-2s}
    \int_{|z|>1}\frac{\log\left(3+|z|\right)}{|z|^{n+2s}}dz\notag\\  
    &\leq C\left(\log\frac{1}{\e}\right)^{-\frac1p}\e^{-2s}.\notag
\end{align}
Finally, for $m=0$, using that $|v_\e|\leq C\left(\log\frac{1}{\e}\right)^{-\frac1p}$ on $B_\frac18^c$, we bound
$$|I_3|\leq C\left(\log\frac{1}{\e}\right)^{-\frac1p}\int_{|y|\geq\frac14}\frac{dy}{|y|^{n+2s}}\leq C\left(\log\frac{1}{\e}\right)^{-\frac1p}.$$
The lemma follows immediately.

\end{proof}

\begin{lem}\label{f-epsilon}
For $|x|\geq 3\e$ we have 
 \begin{align}\notag
  |\D^\sigma f_\varepsilon(x)|\leq C\frac{1}{|x|^{2\sigma}}\left(\log\frac 1\varepsilon\right)^{-\frac{1}{p}}
  \left\{\begin{array}{ll}
       \left(\frac{\e}{|x|}\right)^n &\text{ if }0<\sigma<1\\
       \rule{.0cm}{.7cm}
      \left(\frac{\e}{|x|}\right)^{n-2m} &\text{ if }1<\sigma=m+s<\frac n2,
                                                             \end{array}
\right.
 \end{align}
where $m$ is a positive integer and $0<s<1$. In particular 
 $$\|(-\Delta)^\frac{n}{2p}R_\e\|_{L^p(B_2\setminus B_{3\e})}\leq C\left(\log\frac 1\e\right)^{-\frac 1p}.$$
\end{lem}
\begin{proof}
Notice that for every nonzero multiindex $\alpha\in\mathbb{N}^n$ we have
$$|D^\alpha f_\varepsilon(x)|\leq C\left(\log\frac 1\varepsilon\right)^{-\frac{1}{p}}\left\{\begin{array}{ll}
                            \frac{1}{|x|^{|\alpha|}}&\text{ if }|x|<\varepsilon\\
                            \frac{1}{\varepsilon^{|\alpha|}}                  &\text{ if }\varepsilon<|x|\leq 2\varepsilon\\
                            0 &\text{ if }|x|\geq 2\varepsilon.
                           \end{array}\right.
$$
 First we consider $0<\sigma<1$. Using that  $|\varphi_\e|\leq 1$, changing the variable $y\mapsto\e y$ and by H\"older inequality we obtain
 \begin{align}
  |(-\Delta)^\sigma f_\varepsilon(x)|&=C\left|\int_{\R^n}\frac{f_\varepsilon(x)-f_\varepsilon(y)}{|x-y|^{n+2\sigma}}dy\right|\notag\\
  &=C\left(\log\frac 1\varepsilon\right)^{-\frac{1}{p}}\left|\int_{|y|<2\varepsilon}\frac{\left(\log\frac{1}{\varepsilon}-
  \log\frac{1}{|y|}\right)\varphi_\varepsilon(|y|)}{|x-y|^{n+2\sigma}}dy\right|\notag\\
  &\leq C\e^n\left(\log\frac 1\e\right)^{-\frac{1}{p}}\left(\int_{|y|<2}\frac{dy}{|x-\e y|^{np+2p\sigma}}\right)^\frac 1p
 \left(\int_{|y|<2}\left|\log|y|\right|^{p'}dy\right)^\frac {1}{p'}\notag\\
 &\leq C\e^n\left(\log\frac 1\e\right)^{-\frac{1}{p}}\left(\frac{|x|^n}{\e^n}\frac{1}{|x|^{np+2p\sigma}}
 \int_{|y|<\frac{2\e}{|x|}}\frac{dy}{|\frac{x}{|x|}- y|^{np+2p\sigma}}\right)^\frac 1p, \notag\\
&\leq C\frac{1}{|x|^{2\sigma}}\left(\frac{\e}{|x|}\right)^n\left(\log\frac 1\e\right)^{-\frac{1}{p}},\notag
  \end{align}
  where in the second last inequality we have used a change of variable $y\mapsto \frac{|x|}{\e}y$ and the last inequality follows from the uniform bound 
  \begin{align}\label{xy}
   \frac{1}{|\frac{x}{|x|}- y|^{np+2p\sigma}}\leq C \text{ for every } |x|\geq 3\e,\,|y|\leq \frac{2\e}{|x|}.
  \end{align}
 For $\sigma>1$, changing the variable $y\mapsto |x|y$ and by \eqref{xy} we have
  \begin{align}
    |(-\Delta)^\sigma f_\varepsilon(x)|&=C\left|\int_{\R^n}\frac{\D^mf_\varepsilon(x)-\D^mf_\varepsilon(y)}{|x-y|^{n+2s}}dy\right|\notag\\
  &=C\left|\int_{|y|<2\e}\frac{\D^mf_\varepsilon(y)}{|x-y|^{n+2s}}dy\right|\notag\\
  & \leq C \left(\log\frac 1\e\right)^{-\frac{1}{p}}\int_{|y|<2\e}\frac{1}{|y|^{2m}}\frac{1}{|x-y|^{n+2s}}dy\notag\\
& \leq C\frac{1}{|x|^{2\sigma}}\left(\frac{\e}{|x|}\right)^{n-2m}\left(\log\frac 1\e\right)^{-\frac{1}{p}}.\notag
  \end{align}
We conclude the lemma by \eqref{est-g}.
\end{proof}

 \begin{lem}\label{V-epsilon}
  For $0<\sigma<\frac n2$ there exists a constant $C=C(n,\sigma)$ such that 
  $$|\D^\sigma v_\e(x)|\leq C\left(\log\frac 1\e\right)^{-\frac 1p}\frac{1}{|x|^{n+2\sigma}}\quad\text{for every }x\in B_2^c.$$
  Moreover, $$\|\D^\frac{n}{2p} v_\e\|^p_{L^p(B_2^c)}\leq C\left(\log\frac 1\e\right)^{-1}.$$
   \end{lem}
\begin{proof}
 If $0<\sigma<1$ then
 \begin{align}
  |\D^\sigma v_\e(x)|&= C\int_{|y|<1}\frac{v_\e(y)}{|x-y|^{n+2\sigma}}dy, \qquad |x|>2\label{int}\\
 & \leq C\frac{1}{|x|^{n+2\sigma}}\int_{|y|<1}v_\e(y)dy\notag\\
 & \leq C\left(\log\frac 1\e\right)^{-\frac 1p}\frac{1}{|x|^{n+2\sigma}}\int_{|y|<1}\left(\log|y|+\log2\right)dy\notag\\
 &\leq C\left(\log\frac 1\e\right)^{-\frac 1p}\frac{1}{|x|^{n+2\sigma}}.\notag
 \end{align}
Since the integral in the right hand side of \eqref{int} is a proper integral, differentiating under the integral sign one can prove the 
lemma in a similar way.
\end{proof}

\noindent{\emph{Proof of Theorem \ref{main-thm}}}
Without loss of generality we can assume that $B_1\subseteq\Omega.$ Let $u_\e$ be defined as in Proposition \ref{prop1}.
We claim that there exists a constant $\delta>0$ such that 
\begin{align}\label{lower}
 \limsup_{\e\to 0}\int_{B_\e}\exp\left(\frac{\alpha_{n,p}|u_\e|^{p'}}{\left(\|u_\e\|^p_{L^p(\R^n)}+\|\D^\frac{n}{2p}u_\e\|^p_{L^p(\R^n)}
 \right)^{\frac{p'}{p}}}\right)dx=:\limsup_{\e\to 0}I_\e\geq\delta.
\end{align}
Then Theorem \ref{main-thm} follows at once, since $u_\e\to\infty$  on $B_\e$ as $\e\to0$ and
$$\sup_{u\in \tilde{H}^{\frac np,p}(\Omega),\,\|u\|^p_{L^p(\Omega)}+\|(-\Delta)^\frac{n}{2p}u\|_{L^p(\R^n)}^p\leq 1}
\int_{\Omega}f(|u|)e^{\alpha_{n,p}|u|^{p'}}dx\geq I_\e\inf_{x\in B_\e}f(|u_\e(x)|).$$

To prove \eqref{lower} we choose $\e=e^{-k}$. Noticing that 
\begin{align}
 \lim_{k\to\infty}-k+k\left(1+\frac Ck\right)^{-\frac{p'}{p}} =-C\frac{p'}{p},\notag
\end{align}
$$\|u_\e\|^p_{L^p(\R^n)}\leq C\left(\log\frac1\e\right)^{-1},$$
and using Proposition \ref{prop1} we have
\begin{align}
I_\e
\geq |B_1|\e^ne^{n\log\frac{1}{\e}\left(1+C\left(\log\frac{1}{\e}\right)^{-1}\right)^{-\frac{p'}{p}}}
=|B_1|e^{-kn+kn\left(1+\frac Ck\right)^{-\frac{p'}{p}}}
\geq \delta,\notag
\end{align}
for some $\delta>0$.
\hfill $\square$
\medskip

In order to prove Theorem \ref{thm-2}, first we prove the following proposition which gives a similar type of estimate as in Proposition \ref{prop1}.
\begin{prop}\label{prop2}
 Let $\tau>0$ and $1<p<\infty$. Then there exists a constant $C>0$ such that 
 $$\|(\tau I-\Delta)^{\frac{n}{2p}}u_\e\|_{L^p(\R^n)}\leq \left(1+C\left(\log\frac 1\e\right)^{-1}\right)^\frac1p.$$
\end{prop}
\begin{proof}
We set $$w_\e(x)=(\tau I-\Delta)^{\frac{n}{2p}}u_\e(x)-\D^\frac{n}{2p}u_\e(x).$$
We observe that there exists $C=C(p)>0$ such that 
$$h(t)=(1+t)^p-1-C(t^p+t^{p-1}+t^\frac12)<0, \quad\text{for every }t>0,\,1\leq p<\infty,$$
which follows from the fact that $h(0)=0$ and $h'(t)<0$ for every $t>0$. Therefore, there holds
 $$(a+b)^{p}\leq a^{p}+C_p(b^{p}+ab^{p-1}+b^{\frac12}a^{p-\frac12}),\quad a\geq0,\,b\geq0,\,1\leq p<\infty,$$
 for some constant $C_p>0$ and using this inequality  we bound
 \begin{align}
 & \int_{\R^n}|(\tau I-\Delta)^{\frac{n}{2p}}u_\e(x)|^pdx\notag\\
  &=\int_{\R^n}|w_\e(x)+\D^\frac{n}{2p}u_\e(x)|^pdx\notag\\
  &\leq\int_{\R^n}|\D^\frac{n}{2p}u_\e(x)|^p+C\int_{\R^n}|w_\e(x)|^pdx+C\int_{\R^n}|\D^\frac{n}{2p}u_\e(x)||w_\e(x)|^{p-1}dx\notag\\
  &\quad +C\int_{\R^n}|\D^\frac{n}{2p}u_\e(x)|^{p-\frac12}|w_\e(x)|^{\frac12}dx\notag\\
  &=:I_1+I_2+I_3+I_4. \notag
 \end{align}
From Proposition \ref{prop1} we have $$I_1\leq 1+C\left(\log\frac1\e\right)^{-1}.$$ 
To estimate $I_2$, $I_3$ and $I_4$ we will use pointwise estimates on $\D^\sigma u_\e$, $\D^\sigma w_\e$ and $L^p$ estimates on $\D^\sigma w_\e$.
For $0<\sigma<\frac n2$ 
combining Lemmas \ref{V-near-zero} - \ref{V-epsilon},  \ref{log}, and \eqref{est-g} we get
\begin{align}\label{u-eps}
 |\D^\sigma u_\e(x)|\leq C\left(\log\frac 1\e\right)^{-\frac 1p}\left\{\begin{array}{ll}
                              \e^{-2\sigma} &\text{ if }|x|<3\e\\
                               |x|^{-2\sigma} &\text{ if }3\e<|x|<2\\
                               |x|^{-n-2\sigma} &\text{ if }|x|>2.
                                                              \end{array}                                                                                                                    
                                \right.
\end{align}
With the help of  \eqref{u-eps} one can verify that
\begin{align}\label{lps}
 \|\D^\sigma u_\e\|_{L^p(\R^n)}\leq C(n,p,\sigma)\left(\log\frac1\e\right)^{-\frac 1p},\quad 1\leq p<\infty,\, 0\leq\sigma<\frac{n}{2p},
\end{align}
 and together with Lemma \ref{Lp}
 $$I_2\leq C\left(\log\frac1\e\right)^{-1}.$$
We conclude the proposition by showing that
\begin{align}
 \int_{\R^n}|w_\e|^q|\D^\frac{n}{2p}v_\e|^{p-q}dx\leq C(n,p,q)\left(\log\frac 1\e\right)^{-1},\quad 0< q<\frac{p^2}{p+1}.\label{show}
\end{align}
It follows from Lemma \ref{pointwise} that 
 $$|w_\e(x)|\leq C\left(\log\frac 1\e\right)^{-\frac 1p},\quad x\in\R^n,\,\frac{n}{2p}<1,$$ 
 and for $\frac{n}{2p}>1$
$$|w_\e(x)|\leq C\left(\log\frac 1\e\right)^{-\frac 1p}\left\{\begin{array}{ll}
                              \e^{-\frac np+2} &\text{ if }|x|<3\e\\
                               |x|^{-\frac np+2} &\text{ if }3\e<|x|<2\\
                               1 &\text{ if }|x|>2,
                                                              \end{array}                                                                                                                              
                                \right.$$
   thanks to \eqref{u-eps} and \eqref{lps}.  
   
    Splitting $\R^n$ into 
$$A_1=\left\{x:|x|\leq 2\right\}\quad \text{ and }A_3=\left\{x:|x|>2\right\},$$ we have
       $$\int_{\R^n}|w_\e|^q|\D^\frac{n}{2p}v_\e|^{p-q}dx=\sum_{i=1}^2J_i,\quad J_i:=\int_{A_i}|w_\e|^q|\D^\frac{n}{2p}v_\e|^{p-q}dx,\,i=1,2.$$
  Using \eqref{u-eps} one can show that 
   $J_1\leq C\left(\log\frac 1\e\right)^{-1}$ and together with  $q<\frac{p^2}{p+1}$ one has
$ J_3 \leq C\left(\log\frac 1\e\right)^{-1},$ which gives \eqref{show}. 
\end{proof}

\noindent{\emph{Proof of Theorem \ref{thm-2}}} 
Here also we can assume that $B_1\subseteq\Omega$.
We choose $M>0$ large enough such that 
$$\Phi(\alpha_{n,p}t^{p'})\geq \frac 12 e^{\alpha_{n,p}t^{p'}},\quad t\geq M.$$
Then we have 
\begin{align}
 &\int_{\R^n}f(|u_\e|)\Phi\left(\alpha_{n,p}|u_\e|^{p'}\|(\tau I-\Delta)^\frac{n}{2p}u_\e\|_{L^p(\R^n)}^{-p'}\right)dx\notag\\
 &\geq \int_{u_\e\geq M}f(|u_\e|)\Phi\left(\alpha_{n,p}|u_\e|^{p'}\|(\tau I-\Delta)^\frac{n}{2p}u_\e\|_{L^p(\R^n)}^{-p'}\right)dx\notag\\ 
 &\geq\frac 12\int_{B_\e}f(|u_\e|)e^{\alpha_{n,p}|u_\e|^{p'}\|(\tau I-\Delta)^\frac{n}{2p}u_\e\|_{L^p(\R^n)}^{-p'}}dx,\notag
\end{align}
for $\e>0$ small enough. Now the proof follows as in  Theorem \ref{main-thm}, thanks to Proposition \ref{prop2}.
\hfill $\square$


\section{Proof of Theorem \ref{thm-3}}
Throughout this section we use the notation $\|u\|=\|\D^ \frac n4u\|_{L^2(\R^n)}$, $H=\tilde{H}^{\frac n2,2}(\Omega)$ and $\alpha_0=\alpha_{n,2}$.

To prove Theorem \ref{thm-3} we follow the approach in \cite{Adi, Lakkis}. First we prove that $\lambda_1>0$, which makes the statement of Theorem \ref{thm-3}
meaningful.

\begin{lem}
 Let $\Omega$ be an open set in $\R^n$ with finite measure. 
 Then $\lambda_1>0$ and there exists
 a function $u\in H$ such that $$\|u\|_{L^2(\Omega)}=1,\,\text{ and }\| u\|^2=\lambda_1.$$
\end{lem}
\begin{proof}
We recall that 
\begin{align}
\lambda_1=\inf\left\{\|u\|^2:u\in H,\,\|u\|_{L^2(\Omega)}=1\right\}.\notag
\end{align}
 Let $\{u_k\}_{k=1}^\infty\subset H$ be a sequence such that  
 $$\lim_{k\to\infty}\|u_k\|^2=\lambda_1,\quad \|u_k\|_{L^2(\Omega)}=1\text{ for every }k.$$
 Then up to a subsequence 
 $$u_k\rightharpoonup u_0\text{ in }H,\quad u_k \rightarrow u_0\text{ in }L^2(\Omega),$$ 
 where the latter one follows from the compact embedding $H \hookrightarrow L^2(\Omega)$ (see Lemma \ref{compact}).
  Therefore, 
 $$\lambda_1\leq \|u_0\|^2\leq \liminf_{k\to\infty}\|u_k\|^2=\lambda_1,\quad \|u_0\|_{L^2(\Omega)}=1.$$
\end{proof}

Let us now define the functional 
$$J(u)=\frac{1}{2}\|u\|^2-\int_{\Omega}G(u)dx,\quad u\in H,$$
where $$G(t)=\int_{0}^tg(r)dr,\quad g(r):=\lambda re^{br^2},\quad 0<\lambda<\lambda_1,\,b>0.$$
Then $J$ is $C^2$ and the Fr\'echet derivative of $J$ is be given by 
$$DJ(u)(v)=\int_{\R^n}\D^\frac n4u\D^\frac n4 vdx-\int_{\Omega}g(u)vdx,\quad v\in H.$$
We also define
$$F(u)=DJ(u)(u)=\|u\|^2-\int_{\Omega}g(u)udx,\quad I(u)=J(u)-\frac12 F(u),$$ 
$$S=\left\{u\in H:u\neq0, F(u)=0\right\}.$$
Observe that if $u\in H$ is a nontrivial weak solution of \eqref{eqn} then $u\in S$.

With the above notations we have:
\begin{lem}\label{s}
The set $S$ is closed in the norm topology and 
 $$0<s^2<\frac{\alpha_0}{b},\quad s:=\sqrt{2\inf_{u\in S}{J(u)}}.$$
\end{lem}
\begin{proof}
 Since $F$ is continuous (actually $F$ is $C^1$ as $J$ is $C^2$), it is enough to show that $0$ is an isolated point of $S$. If not, then there exists a sequence $\{u_k\}\subset S$
 such that $\|u_k\|\to0$ as $k\to\infty$. We set $v_k=\frac{u_k}{\|u_k\|}$. From  the compactness of the embedding $H \hookrightarrow L^q(\Omega)$ for any $1\leq q<\infty$,
 we can assume that (up to a subsequence) $v_k\rightharpoonup v$ in $H$ and $v_k\to v$ almost everywhere in $\Omega$. By Lemma \ref{conv1} (below) we get
 $$1=\lambda\int_{\Omega}e^{bu_k^2}v_k^2dx\xrightarrow{k\to\infty}\lambda\int_{\Omega}v^2dx\leq \lambda\frac{1}{\lambda_1}\|v\|^2<1,$$ which is a contradiction.
 Hence $S$ is closed.
 
 Since, $$f(t):=\left(t^2-\frac1b\right)e^{bt^2}+\frac 1b>0,\quad\text{for }t>0,\,b>0,$$ which follows from $f(0)=0$ and $f'(t)>0$ for $t>0$,
 we have 
 \begin{align}
  I(u)=\frac{\lambda}{2}\int_{\Omega}\left(\left(u^2-\frac 1b\right)e^{bu^2}+\frac 1b\right)dx>0,\text{ if }u\in H\setminus\{0\},\label{Iu}
 \end{align}
  and in particular $J(u)=I(u)>0$ for $u\in S$.
  
 If possible, we assume that $s=0$. Then there exists a sequence $\{u_k\}\subset S$ such that $J(u_k)\to 0$ as $k\to\infty.$ Moreover,
 \begin{align}
  \|u_k\|^2=\lambda\int_{\Omega}u_k^2e^{bu_k^2}dx &= \lambda\int_{u_k^2>\frac 2b}u_k^2e^{bu_k^2}dx+\lambda\int_{u_k^2\leq\frac 2b}u_k^2e^{bu_k^2}dx\notag\\
  &\leq4\frac{\lambda}{2}\int_{u_k^2>\frac 2b}\left(\left(u_k^2-\frac 1b\right)e^{bu_k^2}+\frac 1b\right)dx+\lambda\int_{u_k^2\leq\frac 2b}u_k^2e^{bu_k^2}dx\notag\\
 & \leq4J(u_k)+\lambda\int_{u_k^2\leq\frac 2b}u_k^2e^{bu_k^2}dx, \label{abcd}
 \end{align}
and hence $u_k$ is bounded in $H$. Then up to a subsequence $u_k\rightarrow u,\, a.e.$ in $\Omega$ and $u_k\rightharpoonup u$.
 Using Fatou lemma and $ii)$ in Lemma \ref{conv1} we obtain $$I(u)=\frac{\lambda}{2}\int_{\Omega}\left(\left(u^2-\frac 1b\right)e^{bu^2}+\frac 1b\right)dx\leq \liminf_{k\to\infty} 
 I(u_k)=\liminf_{k\to\infty} J(u_k)=0,$$
 and hence $u=0$, thanks to \eqref{Iu}. It follows from \eqref{abcd} that $u_k\to 0$ in $H$ which is a contradiction as $S$ is closed.
 
 We prove now $s^2<\alpha_0b^{-1}$. First we fix $u\in H$ with  $\|u\|=1$. We consider the function 
 $$F_u(t):=F(tu)=\|tu\|^2-\lambda \int_{\Omega}t^2u^2e^{bt^2u^2}dx,\quad t\geq0.$$ Then 
 $$F_u(t)\geq t^2\left(\lambda_1\int_{\Omega}u^2dx-\lambda \int_{\Omega}u^2e^{bt^2u^2}dx\right)>0,$$ for $t>0$ sufficiently small and $\lim_{t\to\infty}F_u(t)=-\infty$.
 Hence, the continuity of $F_u$ implies that there exists $t_u>0$ such that $F_u(t_u)=0$, i.e., $t_uu\in S$. Thus 
 $$\frac{s^2}{2}\leq J(t_uu)\leq \frac12\|t_uu\|^2=\frac12t_u^2.$$ Again using that  $t_uu\in S$ we have
 $$\int_{\Omega}u^2e^{bs^2u^2}dx\leq \frac{1}{\lambda t_u^2}\lambda\int_{\Omega}(t_uu)^2e^{b(t_uu)^2}dx=\frac{1}{\lambda t_u^2}\|t_uu\|^2=\frac1\lambda,$$ which implies that
\begin{equation}\label{estimates}
\sup_{\|u\|\leq 1,\,u\in H}\int_{\Omega}u^2e^{bs^2u^2}dx<\infty,
\end{equation}
and by Theorem \ref{main-thm} we deduce that $s^2<\alpha_0b^{-1}$.
 \end{proof} 
 
 \begin{lem}\label{critical}
 Let $u\in S$ be a minimizer of $J$ on $S$. Then $DJ(u)=0$. 
\end{lem}
\begin{proof}
 We fix a function $v\in H\setminus\{0\}$ and consider the function $$F_{u,v}(\gamma,t):=F(\gamma u+tv),\quad \gamma>0,\,t\in \R.$$
 Differentiating $F_{u,v}$ with respect to $\gamma$  and using that $F(u)=0$, we get
 $$\frac{\partial F_{u,v}}{\partial \gamma}(1,0)=-2b\lambda\int_{\R^n}u^4e^{bu^2}dx<0.$$
 Hence, by implicit function theorem, there exists $\delta>0$ such that we can write $\gamma=\gamma(t)$ as a $C^1$ function of $t$ on the 
 interval $(-\delta,\,\delta)$ which satisfies $$\gamma(0)=1,\quad  F_{u,v}(\gamma(t),t)=0,\text{ for every }t\in(-\delta,\,\delta).$$
 Moreover, choosing $\delta>0$ smaller if necessary, we have $\gamma(t)u+tv\in S$ for every $t\in(-\delta,\,\delta)$. We write
 \begin{align}
  DJ(u)(v)&=\lim_{t\to0}\frac{J(u+tv)-J(u)}{t}\notag\\
 & =\lim_{t\to0}\left(\frac{J(\gamma(t)u+tv)-J(u)}{t}-\frac{J(\gamma(t)u+tv)-J(u+tv)}{t}\right).\notag
 \end{align}
 Since $J$ is $C^1$, a first order expansion of $J$ yields 
\begin{align}
 J(\gamma(t)u+tv)-J(u+tv)&=J((u+tv)+(\gamma(t)-1)u)-J(u+tv)\notag\\
 &=DJ(u+tv)((\gamma(t)-1)u)+o\left((\gamma(t)-1)\|u\|\right)\notag\\
 &=(\gamma(t)-1)DJ(u+tv)(u)+(\gamma(t)-1)\|u\|o(1).\notag
\end{align}
Therefore, using that $F(u)=0$,
$$\lim_{t\to0}\frac{J(\gamma(t)u+tv)-J(u+tv)}{t}=\gamma'(0)DJ(u)(u)=0.$$
On the other hand, since $u$ is a minimizer of $J$ on $S$ and $\gamma(t)u+tv\in S$,
$$\frac{J(\gamma(t)u+tv)-J(u)}{t}=\left\{\begin{array}{ll}
                                          \geq 0 &\text{ if }t\geq0\\
                                          \leq 0 &\text{ if }t\leq 0,
                                         \end{array}\right.
$$
implies that (since it exists) $$\lim_{t\to0}\frac{J(\gamma(t)u+tv)-J(u)}{t}=0.$$ This shows that $DJ(u)(v)=0$ for every $v\in H$, i.e., $DJ(u)=0$.
\end{proof}

\medskip

\noindent
{\emph{Proof of Theorem \ref{thm-3}}}
Let $\{u_k\}$ be a sequence in $S$ such that $\lim_{k\to\infty}J(u_k)\to \frac{s^2}{2}$. Then by \eqref{abcd} $u_k$ is a bounded sequence in $H$ and consequently,
up to a subsequence $$u_k\rightharpoonup u,\quad u_k\rightarrow u,\,a.e. \text{ in }\Omega,\quad \ell:=\lim_{k\to\infty}\|u_k\|,$$ for some $u\in H$.
First we claim that $u\neq0$. 

Assuming $u=0$, by $ii)$ in Lemma \ref{conv1} (below) we get 
$$\lim_{k\to\infty}\|u_k\|^2=\lim_{k\to\infty}2\left(J(u_k)+\frac{\lambda}{2b}\int_{\Omega}(e^{bu_k^2}-1)dx\right)=s^2<\frac{\alpha_0}{b},$$
and hence by $i)$ in Lemma \ref{conv1}
$$\lim_{k\to\infty}\|u_k\|^2=\lim_{k\to\infty}\lambda\int_{\Omega}u_k^2e^{bu_k^2}dx=0,$$ a contradiction as $S$ is closed. 

We claim that  $\ell=\|u\|$. Then $u_k\rightarrow u$ in $H$ and applying Lemmas \ref{s} and \ref{critical} we have Theorem \ref{thm-3}.

If the claim is false then necessarily we shall have $\ell>\|u\|$. 

One has
\begin{align}
\lim_{k\to\infty}\|u_k\|^2&=\lim_{k\to\infty}2\left(J(u_k)+\frac{\lambda}{2b}\int_{\Omega}(e^{bu_k^2}-1)dx\right)\notag\\
&=2\left(\frac{s^2}{2}+\frac{\lambda}{2b}\int_{\Omega}(e^{bu^2}-1)dx,\right)\notag\\
&=s^2-2J(u)+\|u\|^2.\notag
\end{align}
We divide the proof in two cases, namely $J(u)\leq 0$ and $J(u)>0$.

\noindent\textbf{Case 1.} We consider that $J(u)\leq 0$. Since $u\neq 0$,
$$\|u\|^2\leq \frac{\lambda}{b}\int_{\Omega}(e^{bu^2}-1)dx<\lambda\int_{\Omega}u^2e^{bu^2}dx,$$ where the second inequality follows
from \eqref{Iu}. It is easy to see that we can choose $0<t_0<1$ such that
$$\|t_0u\|^2=\lambda\int_{\Omega}(t_0u)^2e^{b(t_0u)^2}dx,$$ that means $t_0u\in S$. Using that  $I(tu)$ is strictly monotone increasing in $t$, which follows
from the expression in \eqref{Iu}, we obtain
\begin{align}
 \frac{s^2}{2}\leq J(t_0u)=I(t_0u)<I(u)\leq \liminf_{k\to\infty}J(u_k)=\frac{s^2}{2},\notag
\end{align}
a contradiction.

\noindent\textbf{Case 2.} Here we assume that $J(u)> 0$. Then 
\begin{align}
 \ell^2=\lim_{k\to\infty}\|u_k\|^2=s^2-2J(u)+\|u\|^2<s^2+\|u\|^2<\frac{\alpha_0}{b}+\|u\|^2.\label{ell}
\end{align}
 Taking $v_k=\frac{u_k}{\|u_k\|}$ we see that (up to a subsequence)
$$v_k\rightharpoonup v:=\frac {u}{{\ell}},\quad v_k\rightarrow v,\,a.e.\text{ in }\Omega,$$ and by Lemma \ref{lions}, for every $p<(1-\|v\|^2)^{-1}$
$$\sup_{k\in\mathbb{N}}\int_{\Omega}e^{p\alpha_0v_k^2}dx<\infty.$$
Taking \eqref{ell} into account we have 
$$0<\ell^2-\|u\|^2=s^2-2J(u)<\frac{\alpha_0}{b},$$ and therefore, we can choose $\e_0>0$ such that
$$1+\e_0=\frac{\alpha_0}{b}\frac{1}{\ell^2-\|u\|^2},\quad\text{i.e., }\ell^2(1+\e_0)=\frac{\alpha_0}{b}\left(1-\frac{\|u\|^2}{\ell^2}\right)^{-1}.$$
For $k$ large enough such that $\|u_k\|^2\leq \ell^2(1+\frac{\e_0}{2})$ holds, we observe that 
$b\|u_k\|^2\leq p_0\alpha_0$ for some $1<p_0<(1-\|v\|^2)^{-1}$. Thus, for some $p_1>1$, $p_2>1$ with $p_1p_2p_0<(1-\|v\|^2)^{-1}$ we obtain
$$\sup_{k\in\mathbb{N}}\int_{\Omega}\left(u_k^2e^{bu_k^2}\right)^{p_1}dx\leq \sup_{k\in\mathbb{N}}\|u_k^{2p_1}\|_{L^{p_2'}(\Omega)}\|e^{p_1p_0\alpha_0v_k^2}\|_{L^{p_2}(\Omega)}<\infty, $$
and together with Lemma \ref{epsilon-conv}
$$\lim_{k\to\infty}\int_{\Omega}u_k^2e^{bu_k^2}dx=\int_{\Omega}u^2e^{bu^2}dx.$$
Indeed,
$$\|u\|^2<\ell^2=\lim_{k\to\infty}\|u_k\|^2=\lambda\lim_{k\to\infty}\int_{\Omega}u_k^2e^{bu_k^2}dx=\lambda\int_{\Omega}u^2e^{bu^2}dx,$$ 
and  we can now proceed as in Case 1.
\hfill $\square$

\begin{lem} \label{conv1}
 Let $u_k,\,v_k\in H$ such that $u_k\rightharpoonup u \text{ in }H$, $ u_k\rightarrow u, \text{ a.e. in }\Omega,$  
 $v_k\rightharpoonup v\text{ in }H$ and $ v_k\rightarrow v, \text{ a.e. in }\Omega.$ Then 
 \begin{itemize}
  \item [i)] If $$\limsup_{k\to\infty}\|u_k\|^2<\frac{\alpha_0}{b},$$ then for every integer $\ell\geq 1$
  $$\lim_{k\to\infty}\int_{\Omega}e^{bu_k^2}v_k^\ell dx=\int_{\Omega}e^{bu^2}v^\ell dx.$$
  \item [ii)] If $$\limsup_{k\to\infty}\int_{\Omega}u_k^2e^{bu_k^2}dx<\infty,$$ then
  $$\lim_{k\to\infty}\int_{\Omega} e^{bu_k^2}dx=\int_{\Omega} e^{bu^2}dx.$$
 \end{itemize}
\end{lem}
\begin{proof}
We prove the lemma with the help of Lemma \ref{epsilon-conv} (in Appendix).

 We choose $p>1$ such that for $k$ large enough  $p\|u_k\|^2<\frac{\alpha_0}{b}$ holds and together with Theorem C we have
 $$\sup_{k\in\mathbb{N}}\int_{\Omega}e^{pbu_k^2}dx<\infty.$$ Since the embedding $\tilde{H}^{\frac n2,2}(\Omega) \hookrightarrow L^q(\Omega)$ is compact (see Lemma \ref{compact}) for every
 $1\leq q<\infty$, we have $$v_k^q\rightarrow v^q\text{ in }L^1(\Omega).$$ Indeed,
 $$\sup_{k\in\mathbb{N}}\|e^{bu_k^2}v_k^\ell\|_{L^p(\Omega)}\leq \|v_k^\ell\|_{L^{p'}(\Omega)}\|e^{bu_k^2}\|_{L^p(\Omega)}<\infty,$$
 and we conclude $i)$.
 
  Now $ii)$ follows from $$\int_{u_k^2>M}e^{bu_k^2}dx\leq \frac 1M\int_{u_k^2>M}u_k^2e^{bu_k^2}dx\leq \frac CM,$$ 
  which implies that the function $f_k:=e^{bu_k^2}$ satisfies the condition  $ii)$ in Lemma \ref{epsilon-conv}.
\end{proof}

In the following lemma we prove that  the assumption $H'(v)$ in \cite{IS} is true under certain conditions. 
 \begin{lem}\label{Ian-Squa}
 Let $\alpha_0>0$. Let $f(t)=e^{\alpha_0t^2}h(t)$ satisfies $H(i)-(iii)$ in \cite{IS}. Let $h\geq0$ on $[0,\infty)$ and $h(-t)=-h(t)$. Let 
 $s\frac{f(st)}{t}$ be a monotone increasing function with respect to $t$ on $(0,\infty)$, $s\neq0$. If
  $\lim_{t\to\infty}h(t)t=\infty$ then there exists $u\in \tilde{H}^{\frac12,2}((0,1))$ such that $\sqrt{2\pi}\|\D^\frac14u\|_{L^2(\R)}=1$ and
   $$\sup_{t>0}\varPhi(tu):=\sup_{t>0}\left(\frac{t^2}{4\pi}-\int_0^1F(tu)dx\right)<\frac{\omega}{2\alpha_0},$$
   where $$F(t)=\int_0^tf(s)ds,$$ and $\omega$ is as in \cite{IS}.
  \end{lem}
\begin{proof}
For a given $M>0$ we can choose $u\in \tilde{H}^{\frac12,2}((0,1))$ such that 
$$\int_0^1f\left(\sqrt{\frac{2\pi^2}{\alpha_0}}u\right)udx>M,\quad \sqrt{2\pi}\|\D^\frac14u\|_{L^2(\R)}=1,$$
thanks to Theorem \ref{main-thm}. Differentiating with respect to $t$ one has
$$\varPhi'(tu)=t\left(\frac{1}{2\pi}-\int_0^1\frac{f(tu)}{t}udx\right).$$ Hence, for $t\geq \sqrt{\frac{2\pi^2}{\alpha_0}}=:t_0$ and $2\pi M>t_0$
\begin{align}
 \varPhi'(tu)\leq t\left(\frac{1}{2\pi}-\int_0^1\frac{f(t_0u)}{t_0}udx\right)<0. \notag
\end{align}
Thus $\varPhi'(tu)\leq0$ on $(t_0-\e,\infty)$ for some $\e>0$ and therefore,
$$\sup_{t>0}\varPhi(tu)=\sup_{t\in(0,\,t_0-\e)}\varPhi(tu)\leq\sup_{t\in(0,\,t_0-\e)}\frac{t^2}{4\pi}<\frac{\pi}{2\alpha_0}.$$ 
Since $\omega=\pi$, thanks to Theorem  B, we conclude the lemma.
\end{proof}

\appendix
\section{Appendix}
 
 \begin{lem}[Pointwise estimate]\label{pointwise}
Let $s>0$ and not an integer.  Let $m$  be the smallest integer greater than $s$.  Then for any  $\tau>0$
  $$|(\tau I-\Delta)^{s}u(x)-\D^{s}u(x)|\leq C\sum_{j=1}^{m-1}|\D^{s-j}u(x)|+C\|\D^{\sigma}u\|_{{L^{1}(\R^{n})}},\quad u\in\s(\R^{n}),$$
  where $\sigma\in\left(\max\{\frac n2-m+s,0\}, \,\frac n2\right)$, the constant $C$ depends only on $n$, $s$, $\sigma$, $\tau$ and for $m=1$ the 
  above sum can be interpreted as zero.
\end{lem}
\begin{proof}
We set $f(t)=t^{s}$ on $\R^{+}$.  By Taylor's expansion we have
  $$ f(t+\tau)=f(t)+\tau f'(t)+\dots +\frac{\tau^{m-1}}{(m-1)!}f^{m-1}(t)+\frac{\tau^{m}}{m!}f^{m}(\xi_{t}),\quad\text{for some }t<\xi_{t}<t+\tau.$$  In particular
  $$(\tau+t^{2})^{s}=t^{2s}+c_{1}t^{2s-2}+c_{2}t^{2s-4}+\dots +c_{m-1}t^{2s-2m+2}+E(t),$$ where the function $E$ satisfies the estimate
  $$|E(t)|\leq C(1+t)^{2s-2m},\quad t>0.$$  
Therefore, for $u\in\s(\R^{n})$
\begin{align}
\F((\tau I-\Delta)^s u)(\xi)&=(\tau+|\xi|^2)^s\hat{u}\notag\\&=\left(|\xi|^{2s}+c_{1}|\xi|^{2s-2}+\dots +c_{m-1}|\xi|^{2s-2m+2}+E(|\xi|)\right)\hat{u}\notag\\
&=\sum_{j=0}^{m-1}c_{j}|\xi|^{2s-2j}\hat{u}+E(|\xi|)\hat{u}(\xi)\notag \\
&=\sum_{j=0}^{m-1}c_{j}\F(\D^{s-j}u)+E(|\xi|)\hat{u}(\xi),\notag
\end{align}
and hence 
$$(\tau I-\Delta)^{s}u(x)=\sum_{j=0}^{m-1}c_{j}\D^{s-j}u(x)+\mathcal{F}^{-1}(E\hat{u})(x).$$
To estimate the term $\mathcal{F}^{-1}(E\hat{u})$ (uniformly in $x$)  in terms of $L^{1}(\R^{n})$ norm of  (fractional) derivative of $u$, we observe
that 
\begin{align}
|E(|\xi|)\hat{u}(\xi)|&=\left|E(|\xi|)\frac{1}{|\xi|^{2\sigma}}\widehat{\D^{\sigma}u}(\xi)\right|\notag\\&\leq \frac{C}{|\xi|^{2\sigma}(1+|\xi|^2)^{m-s}}\left|\widehat{\D^{\sigma}u}(\xi)\right|\notag\\
&\leq \frac{C}{|\xi|^{2\sigma}(1+|\xi|^2)^{m-s}}
\|\D^{\sigma}u\|_{{L^{1}(\R^{n})}}.\notag
\end{align}
 Thus 
$$\left|\mathcal{F}^{-1}(E\hat{u})(x)\right|\leq C\|E\hat{u}\|_{{L^{1}(\R^{n})}}\leq C\|\D^{\sigma}u\|_{{L^{1}(\R^{n})}},$$ and we complete the proof.
\end{proof}

\begin{lem}[$L^{p}$ Estimate]\label{Lp}
Let $s>0$ be a noninteger. Let $\tau>0$ be any fixed number. Then for $p\in(1,\,\infty)$ there exists $C=C(n,s,p,\tau)>0$ such that 
$$\|(\tau I-\Delta)^{s}u-\D^{s}u\|_{{L^{p}(\R^{n})}}\leq C\left\{\begin{array}{ll}
\|u\|_{{L^{p}(\R^{n})}} & if s<1\\
\|u+\D^{s-1}u\|_{{L^{p}(\R^{n})}} & if s>1.
\end{array}\right.$$
\end{lem}
\begin{proof}
We have
\begin{align}
\F((\tau I-\Delta)^{s}u)(\xi)-\F(\D^{s}u)(\xi)&=\left((\tau+|\xi|^{2})^{s}-|\xi|^{2s}\right)\hat{u}(\xi)\notag\\
&=\left\{\begin{array}{ll}\notag
\left((\tau+|\xi|^{2})^{s}-|\xi|^{2s}\right)\hat{u}(\xi) &\text{ if }s<1\\
\\
\frac{(\tau+|\xi|^{2})^{s}-|\xi|^{2s}}{1+|\xi|^{2s-2}}(1+|\xi|^{2s-2})\hat{u}(\xi) &\text{ if }s>1
\end{array}\right.\\
&=:\left\{\begin{array}{ll}\notag
m(\xi)\hat{u}(\xi) &\text{ if }s<1\\
\\
m(\xi)\mathcal{F}\left(u+\D^{s-1}u\right)(\xi)&\text{ if }s>1.
\end{array}\right.
\end{align}
Now the proof follows from the Hormander multiplier theorem (see \cite[p. 96]{Stein}).
\end{proof}

The following lemma appears already in \cite[p. 46]{FM}, but for the reader's convenience we give a more detailed proof.
\begin{lem}[Equivalence of norms]\label{norm-equiv}
 Let $\sigma>0$. Then for $p\in(1,\infty)$ there exists a constant $C>0$ such that for every $u\in \s(\R^n)$
 \begin{align}
  \frac 1C\left(\|u\|_{L^p(\R^n)}+\|\D^\sigma u\|_{L^p(\R^n)}\right)&\leq\|(I-\Delta)^\sigma u\|_{L^p(\R^n)}\notag\\
  &\leq C\left(\|u\|_{L^p(\R^n)}+\|\D^\sigma u\|_{L^p(\R^n)}\right). \notag
 \end{align}
\end{lem}
\begin{proof}
 We set $$G_\sigma(x)=\frac{1}{(4\pi)^\frac \sigma2}\frac{1}{\Gamma(\frac \sigma2)}\int_0^\infty e^{-\pi\frac{|x|^2}{t}}e^{-\frac{t}{4\pi}}
 t^{\frac{-n+\sigma}{2}}\frac{dt}{t},$$ which is the Bessel potential of order $\sigma$ (see \cite[p. 130]{Stein}). Then 
 $$\int_{\R^n}G_\sigma(x)dx=1,\quad\hat{G}_\sigma(x)=\frac{1}{(2\pi)^\frac n2}\frac{1}{(1+|x|^2)^\frac\sigma2}.$$ Setting $f=(I-\Delta)^\sigma u$ we can write
 $u=G_{2\sigma}*f$ and by Young's inequality one has $\|u\|_{L^p(\R^n)}\leq \|f\|_{L^p(\R^n)}$. Again writing  $u=G_{2\sigma}*f$
 and taking Fourier transform we obtain 
 $$\F(\D^\sigma u)=|\xi|^{2\sigma}\hat{u}=|\xi|^{2\sigma}\frac{1}{(1+|\xi|^2)^\sigma}\hat{f}=:m(\xi)\hat{f},$$ and by Hormander
 multiplier theorem we get $\|\D^\sigma u\|_{L^p(\R^n)}\leq C\|f\|_{L^p(\R^n)}$. Thus, 
 $$\|u\|_{L^p(\R^n)}+\|\D^\sigma u\|_{L^p(\R^n)}\leq C\|(I-\Delta)^\sigma u\|_{L^p(\R^n)}.$$
 To conclude the lemma, it is sufficient to show that 
 \begin{align}
\|\D^s u\|_{L^p(\R^n)}\leq C(n,s,\sigma,p)(\|u\|_{L^p(\R^n)}+\|\D^\sigma u\|_{L^p(\R^n)}),\quad 0<s<\sigma, \label{inter}  
 \end{align}
 thanks to Lemma \ref{Lp}. 
 
 In order to prove \eqref{inter} we fix a function $\varphi\in C_c^\infty(B_2)$ such that $\varphi=1$ on $B_1$. Then
 $$\F(\D^s u)=|\xi|^{2s}\hat{u}=|\xi|^{2s}\varphi\hat{u}+|\xi|^{2s}(1-\varphi)\hat{u}=m_1(\xi)\hat{u}+m_2(\xi)\F(\D^\sigma u),$$ where
 $m_1(\xi)=|\xi|^{2s}\varphi(\xi)$, $m_2(\xi)=|\xi|^{2s-2\sigma}(1-\varphi(\xi))$ are multipliers and we conclude \eqref{inter}
 by Hormander multiplier theorem.
\end{proof}

\begin{lem}[Embedding to an Orlicz space]\label{orlicz}
 Let $\Omega$ be an open set with finite measure. Then for every $u\in\tilde{H}^{\frac n2,2}(\Omega)$ 
 $$\int_{\Omega}e^{u^2}dx<\infty.$$
\end{lem}
\begin{proof} 
We set $f=\D^\frac n4u$. By \cite[Proposition 8]{LM} we have
$$u(x)=\int_{\Omega}G(x,y)f(y)dy,\quad 0\leq G(x,y)\leq\frac{C_n}{|x-y|^\frac n2},$$ where $G$ is a Greens function. 

We choose $M>0$ large enough such that $\|\tilde{f}\|_{L^{2}}C_{n}<\alpha_{0}$, where $\tilde{f}=f-f\chi_{\{|f|\leq M\}}$.
Then 
$$|u(x)|\leq C(M)+C_{n}I_{\frac n2}\tilde{f}(x),\quad I_{\frac n2}\tilde{f}(x):=\int_{\Omega}\frac{|\tilde{f}(y)|}{|x-y|^{\frac n2}}dy,$$ 
and by \cite[Theorem 2]{Adams} we  conclude the proof. 
\end{proof}

As a consequence of the above lemma one can prove  a higher dimensional generalization of  Lions lemma \cite{Lions} 
(for a simple proof see e.g. \cite[Lemma 2.6]{IS}), namely
\begin{lem}[Lions]\label{lions}
 Let $u_k$ be a sequence in $\tilde{H}^{\frac n2,2}(\Omega)$ such that 
 $$u_k\rightharpoonup u\text{ in }\tilde{H}^{\frac n2,2}(\Omega),\quad 0<\|\D^ \frac n4u\|_{L^2(\R^n)}<1,\quad \|\D^ \frac n4u_k\|_{L^2(\R^n)}=1.$$
 Then for every $0<p<\left(1-\|\D^ \frac n4u\|^2_{L^2(\R^n)}\right)^{-1}$, the sequence $\left\{e^{\alpha_0pu_k}\right\}_1^\infty$ is bounded in $L^1(\Omega)$. 
\end{lem}

\begin{lem}[Poincar\'e inequality]\label{poincare}
Let $\Omega$ be an open set with finite measure. Then there exists a constant $C>0$ such that 
$$\|u\|_{{L^{2}(\Omega)}}\leq C\|\D^{\frac s2}u\|_{{L^{2}(\R^{n})}},\text{ for every }u\in\tilde{H}^{s,2}(\Omega).$$
\end{lem}
\begin{proof}
We have $$|\hat{u}(\xi)|\leq \frac{1}{(2\pi)^{\frac n2}}\|u\|_{L^{1}(\Omega)}\leq  \frac{1}{(2\pi)^{\frac n2}}|\Omega|^{\frac12}\|u\|_{L^{2}(\Omega)},$$
and hence
\begin{align}
\|u\|^{2}_{{L^{2}(\Omega)}}&=\int_{\R^{n}}|\hat{u}|^{2}d\xi=\int_{|\xi|<\delta}|\hat{u}|^{2}d\xi+\int_{|\xi|\geq\delta}|\hat{u}|^{2}d\xi\notag\\
&\leq \frac{1}{(2\pi)^n}|\Omega|\|u\|^{2}_{{L^{2}(\Omega)}}|B_1|\delta^{n}+\delta^{-2s}\int_{|\xi|\geq\delta}|\xi|^{2s}|\hat{u}|^{2}d\xi\notag\\
&\leq \frac{1}{(2\pi)^n}|\Omega||B_1|\delta^{n}\|u\|^{2}_{{L^{2}(\Omega)}}+\delta^{-2s}\int_{\R^{n}}|\F(\D^{\frac s2}u)(\xi)|^{2}d\xi.   \notag
\end{align}
Choosing $\delta>0$ so that $\frac{1}{(2\pi)^n}|\Omega||B_1|\delta^{n}=\frac12$ we complete the proof. 
\end{proof}

\begin{lem}[Compact embedding]\label{compact}
 Let  $\Omega$ be an open set in $\R^n$ with finite measure. Then the embedding $\tilde{H}^{s,2}(\Omega)\hookrightarrow \tilde{H}^{r,2}(\Omega)$
 is compact for any $0\leq r<s$ (with the notation $\tilde{H}^{0,2}(\Omega)=L^2(\Omega)$).  Moreover, $\tilde{H}^{\frac n2,2}(\Omega)\hookrightarrow L^p(\Omega)$ is compact for any $p\in[1,\infty)$.
\end{lem}
\begin{proof}
We prove the lemma in few steps.

\noindent \textbf{Step 1} 
The embedding $\tilde{H}^{s,2}(\Omega)\hookrightarrow \tilde{H}^{r,2}(\Omega)$ is continuous for any $0\leq r<s$.

With the notation $\D^0u=u$ we see that 
\begin{align}
 \|\D^\frac r2u\|_{L^2(\R^n)}^2&=\int_{\R^n}|\xi|^{2r}|\hat{u}|^2d\xi=\int_{|\xi|\leq 1}|\xi|^{2r}|\hat{u}|^2d\xi+\int_{|\xi|>1}|\xi|^{2r}|\hat{u}|^2d\xi\notag\\
 &\leq \int_{|\xi|\leq 1}|\hat{u}|^2d\xi+\int_{|\xi|>1}|\xi|^{2s}|\hat{u}|^2d\xi\leq \|u\|_{L^2(\Omega)}^2+\|\D^\frac s2u\|_{L^2(\R^n)}^2,\notag
\end{align}
 which is Step 1, thanks to Lemma \ref{poincare}
\medskip
 
 \noindent \textbf{Step 2}
 For a given $s>0$ and a given $\e>0$ there exists $R>0$ such that 
 $$\|u\|_{L^2(\Omega\cap B_R^c)}\leq \e\| u\|_{\tilde{H}^{s,2}(\Omega)},\quad \text{for every }u\in\tilde{H}^{s,2}(\Omega).$$ 
 To prove Step 2 it is sufficient to consider $0<s<1$, thanks to Step 1.

We fix $\varphi\in C_c^\infty(B_2)$ such that $\varphi=1$ on $B_1$ and $0\leq\varphi\leq 1$. Setting $\varphi_r(x)=\varphi(\frac xr)$ we get
\begin{align}
 &\|(1-\varphi_r)u\|_{L^2(\R^n)}^2=\|\F((1-\varphi_r)u)\|_{L^2(\R^n)}^2\notag\\
 &=\int_{|\xi|<R_1}|\F((1-\varphi_r)u)|^2d\xi+\int_{|\xi|\geq R_1}|\F((1-\varphi_r)u)|^2d\xi\notag\\
 &\leq \frac{1}{(2\pi)^n}|B_{R_1}|\left(\int_{\R^n}|(1-\varphi_r)u|dx\right)^2+R_1^{-2s}\int_{|\xi|\geq R_1}|\xi|^{2s}|\F((1-\varphi_r)u)|^2d\xi \notag\\
 &=:I_1+I_2.\notag
 \end{align}
Using that  $supp \,(1-\varphi_r)u\subset \Omega\cap B_r^c$ and by  H\"older inequality we bound
\begin{align}
 I_1\leq \frac{1}{(2\pi)^n}|B_{R_1}||\Omega\cap B_r^c|\int_{\Omega\cap B_r^c}|(1-\varphi_r)u|^2dx\leq \frac{1}{(2\pi)^n}|B_{R_1}||\Omega\cap B_r^c|\|u\|_{L^2(\Omega)}^2.\notag
\end{align}
From \cite[Proposition 3.4]{Valdinoci} we have 
$$\int_{\R^n}|\xi|^{2s}|\hat{u}|^2d\xi=C(n,s)\int_{\R^n}\int_{\R^n}\frac{|u(x)-u(y)|^2}{|x-y|^{n+2s}}dxdy,$$ and hence
\begin{align}
 I_2&\leq R_1^{-2s}\int_{\R^n}|\xi|^{2s}|\F((1-\varphi_r)u)|^2d\xi\notag\\
  &= C_0R_1^{-2s} \int_{\R^n\times\R^n}\frac{((1-\varphi_r(x))u(x)-(1-\varphi_r(y))u(y))^2}{|x-y|^{n+2s}}dxdy\notag\\
  & =C_0R_1^{-2s}\int_{\R^n\times\R^n}\frac{\left((1-\varphi_r(x))(u(x)-u(y))-u(y)(\varphi_r(x)-\varphi_r(y))\right)^2}{|x-y|^{n+2s}}dxdy\notag\\
  &\leq 2C_0R_1^{-2s}\int_{\R^n\times\R^n}\left(\frac{(1-\varphi_r(x))^2(u(x)-u(y))^2}{|x-y|^{n+2s}}+\frac{u^2(y)(\varphi_r(x)-\varphi_r(y))^2}{|x-y|^{n+2s}}\right)dxdy\notag\\
  &\leq 2C_0R_1^{-2s}\int_{\R^n\times\R^n}\frac{(u(x)-u(y))^2}{|x-y|^{n+2s}}dxdy+2C_0R_1^{-2s}\int_{\R^n}u^2(y)\int_{\R^n}\frac{(\varphi_r(x)-\varphi_r(y))^2}{|x-y|^{n+2s}}dxdy\notag\\
  &\leq C_1R_1^{-2s}(\|\D^su|\|_{L^2(\R^n)}^2+\|u\|_{L^2(\Omega)}^2),\notag
\end{align}
where in the  last inequality we have used that 
$$\int_{\R^n}\frac{(\varphi_r(x)-\varphi_r(y))^2}{|x-y|^{n+2s}}dx\leq C,\quad y\in \R^n,\,r\geq1.$$
 Thus we have Step 2 by choosing $R$ so that $|B_{R_1}||\Omega\cap B_R^c|<\frac \e2$ where $C_1R_1^{-2s}=\frac{\e}{2}$.
\medskip
 
 \noindent\textbf{Step 3} The embedding $\tilde{H}^{s,2}(\Omega)\hookrightarrow L^2(\Omega)$ is compact for any $0<s<1$.
 
 Let us consider a bounded sequence $\{u_k\}_{k=1}^\infty$ in $\tilde{H}^{s,2}(\Omega)$. Let  $\varphi$, $\varphi_\ell$ be as in Step 2 (here $\ell\in\mathbb{N}$).
 Then for a fixed $\ell$ the sequence $\{\varphi_\ell u_k\}_{k=1}^\infty$ is bounded in $\tilde{H}^{s,2}(\Omega)$ (the proof is very similar
 to the estimate of $I_2$ in Step 2).
 
 Since the embedding $\tilde{H}^{s,2}(B_r)\hookrightarrow L^2(B_r)$ is compact 
  (see e.g. \cite[Theorem 7.1]{Valdinoci}), there exists a subsequence 
  $\{u_k^1\}_{k=1}^\infty$ such that $\varphi_1u_k^1\rightarrow u^1$ in $L^2(B_2)$. Inductively we will have 
  $\varphi_\ell u_k^\ell\rightarrow u^\ell$ in $L^2(B_{2\ell})$ where $\{u_k^{\ell+1}\}_{k=1}^\infty$ is a subsequence of $\{u_k^{\ell}\}_{k=1}^\infty$ for $\ell\geq1$.
  Moreover, we have
  $u^{\ell+1}=u^\ell$ on $B_\ell$. Setting $u=\lim_{\ell\to\infty}u^\ell$ it follows that $u_k^k$ converges to $u$ in $L^2(\Omega)$, thanks to Step 2.
  \medskip
  
  \noindent\textbf{Step 4} The embedding $\tilde{H}^{s,2}(\Omega)\hookrightarrow \tilde{H}^{r,2}(\Omega)$ is compact for any $0\leq r<s$.
  
  Since the composition of two compact operators is compact, we can assume that $s-r<1$. 
  
  Let $\{u_k\}_{k=1}^\infty$ be a bounded sequence in $\tilde{H}^{s,2}(\Omega)$. Setting $v_k=\D^\frac r2 u_k$ we see that $\{v_k\}_{k=1}^\infty$
  is a a bounded sequence in $\tilde{H}^{s-r,2}(\Omega)$. Then by Step 3 (up to a subsequence) $v_k$ converges to some $v$ in $L^2(\Omega)$ which is 
  equivalent to saying that (up to a subsequence) $u_k$ converges to some $u$ in $\tilde{H}^{r,2}(\Omega)$.
  
  Finally, compactness of the embedding $\tilde{H}^{\frac n2,2}(\Omega)\hookrightarrow L^p(\Omega)$ follows from the compactness of 
  $\tilde{H}^{\frac n2,2}(\Omega)\hookrightarrow L^2(\Omega)$, Theorem B and Lemma \ref{epsilon-conv}.
  \end{proof}

\begin{lem}[Exact constant]\label{log}
 We set $$f(x)=\log\frac{1}{|x|},\quad x\in \R^n.$$ Then
 $$\D^\sigma f(x)=\gamma_n2^{2\sigma-n}\pi^{-\frac n2}\frac{\Gamma(\sigma)}{\Gamma(\frac{n-2\sigma}{2})}\frac{1}{|x|^{2\sigma}},\quad 0<\sigma<\frac n2,$$
 where $\Gamma$ is the  gamma function and $\gamma_n=\frac{(n-1)!}{2}|S^n|.$
\end{lem}
\begin{proof}
 Using a rescaling argument one can get (see for e.g. \cite[Lemma A.5]{H-Structure}) 
 $$\D^\sigma f(x)=\D^\sigma f(e_1)\frac{1}{|x|^{2\sigma}}.$$ To compute the value of $\D^\sigma f(e_1)$ we use the fact that
 $\frac{1}{\gamma_n}\log\frac{1}{|x|}$ is a fundamental solution of $\D^\frac n2$ (see for instance \cite[Lemma A.2]{H-Structure}) i.e.,
 $$\int_{\R^n}\log\frac{1}{|x|}(-\Delta)^\frac n2\varphi(x)dx=\gamma_n\varphi(0),\quad\varphi\in\s(\R^n).$$
 Using integration by parts, which can be verified, we obtain
 \begin{align}
 \gamma_n\varphi(0)&=\int_{\R^n}f(x)(-\Delta)^\frac n2\varphi(x)dx\notag\\
 &=\int_{\R^n}(-\Delta)^\sigma f(x)(-\Delta)^{\frac n2-\sigma}\varphi(x)dx\notag\\
 &=\int_{\R^n}\frac{(-\Delta)^\sigma f(e_1)}{|x|^{2\sigma}}\left(|\xi|^{ n-2\sigma}\widehat{\varphi}\right)^\vee (x)dx\notag\\
 &=(-\Delta)^\sigma f(e_1)\int_{\R^n}\left(\frac{1}{|x|^{2\sigma}}\right)^\vee(\xi)\left(|\xi|^{ n-2\sigma}\widehat{\varphi}\right)d\xi\notag\\
 &=(-\Delta)^\sigma f(e_1)2^{n-2\sigma-\frac n2}\frac{\Gamma(\frac{n-2\sigma}{2})}{\Gamma(\frac{2\sigma}{2})}\int_{\R^n}\frac{1}{|\xi|^{n-2\sigma}}
 \left(|\xi|^{ n-2\sigma}\widehat{\varphi}\right)d\xi\notag\\
&=(-\Delta)^\sigma f(e_1)2^{n-2\sigma-\frac n2}\frac{\Gamma(\frac{n-2\sigma}{2})}{\Gamma(\frac{2\sigma}{2})}(2\pi)^\frac n2\varphi(0),\notag
 \end{align}
 where in the  $4$th equality we have used that  
 \begin{align} \label{constant}
 \F\left(\frac{1}{|x|^{n-\alpha}}\right)=2^{\alpha-\frac n2}\frac{\Gamma\left(\frac\alpha2\right)}{\Gamma\left(\frac{n-\alpha}{2}\right)}\frac{1}{|x|^\alpha},
 \quad 0<\alpha<n, 
 \end{align}
in the sense of tempered distribution. Since in our case $\F$  is the normalized Fourier transform, the constant  in the right hand side of \eqref{constant} appears slightly different from \cite[Section 5.9]{Lieb}.

Hence we have the lemma.
\end{proof}

The following lemma is the  Vitali's convergence theorem.
\begin{lem}[Vitali's convergence theorem]\label{epsilon-conv}
 Let $\Omega$ be a measure space with finite measure $\mu$ i.e., $\mu(\Omega)<\infty$. Let $f_k$ be a sequence of measurable function on $\Omega$ be such that
 \begin{itemize}
  \item [$i)$]$f_k\xrightarrow{k\to\infty} f$ almost everywhere in $\Omega$.
  \item [$ii)$] For $\e>0$ there exists $\delta>0$ such that 
  $$\int_{\tilde{\Omega}}|f_k|d\mu<\e\quad\text{for every }\tilde{\Omega}\subset{\Omega}\text{ with }\mu(\tilde{\Omega})<\delta.$$
 \end{itemize}
 Or,
 \begin{itemize}
  \item [$ii')$] There exists $p>1$ such that $$\sup_{k\in\mathbb{N}}\int_{\Omega}|f_k|^pd\mu<\infty.$$
 \end{itemize}
Then $f_k\rightarrow f$ in $L^1(\Omega)$.  
\end{lem}


\medskip

\noindent\textbf{Acknowledgements}
I would like to thank my advisor Prof. Luca Martinazzi for all the useful discussion and encouragement.

 \end{document}